\documentclass{amsart}
\usepackage{amsmath}
\usepackage{amsfonts}
\usepackage{amssymb}
\usepackage{color}

 \newcommand{\bx}{\mathbf{x}} \newcommand{\by}{\mathbf{y}}
\newcommand{\bm}{\mathbf{m}}\newcommand{\bn}{\mathbf{n}}
\newcommand{\cA}{\mathcal{A}} \newcommand{\cB}{\mathcal{B}}
 \newcommand{\cF}{\mathcal{F}}

\newcommand{\vac}{|0\rangle}

\newtheorem{thm}{Theorem}[section]

\newtheorem{lem}[thm]{Lemma}
\newtheorem{prop}[thm]{Proposition}

\theoremstyle{definition}

\theoremstyle{remark}
\newtheorem{rem}{Remark}[section]

\newcommand{\half}{\frac{1}{2}}

\newcommand{\be}{\begin{equation}}
\newcommand{\ee}{\end{equation}}
\newcommand{\bea}{\begin{eqnarray}}
\newcommand{\eea}{\end{eqnarray}}
\newcommand{\ben}{\begin{eqnarray*}}
\newcommand{\een}{\end{eqnarray*}}
\newcommand{\bt}{\begin{split}}
\newcommand{\et}{\end{split}}
\newcommand{\bet}{\begin{equation}
\begin{split}}
\newcommand{\eet}{\end{split}
\end{equation}}

\DeclareMathOperator{\Span}{span}

\begin{document}

\title[Gromov-Witten invariants of the resolved Conifold]
{On fermionic representation of the Gromov-Witten invariants of the resolved conifold}
\date{}
\author{Fusheng Deng \and Jian Zhou}
\address{Fusheng Deng: \ School of Mathematical Sciences, Graduate University of Chinese Academy of Sciences\\ Beijing 100049, China}
\email{fshdeng@gucas.ac.cn}
\address{Jian Zhou: Department of Mathematical Sciences Tsinghua University\\Beijing, 100084, China}
\email{jzhou@math.tsinghua.edu.cn}

\begin{abstract}
We prove that the fermionic form of the generating function of the Gromov-Witten invariants of the resolved conifold is a
Bogoliubov transform of the fermionic vacuum; in particular, it is a tau function of the KP hierarchy. Our proof is based  on
the  gluing rule of the topological vertex and the formulas of the fermionic representations of the framed one-legged and two-legged
topological vertex which were conjectured by Aganagic et al and proved in our recent work.
\end{abstract}

\maketitle

\section{Introduction}

In general it is an unsolved problem to compute the Gromov-Witten invariants of an algebraic variety
in arbitrary genera.
However,
in the case of toric Calabi-Yau threefolds (which are noncompact),
string theorists have found an algorithm called the topological vertex \cite{AKMV}
to compute the generating function of both open and closed
Gromov-Witten invariants based on a remarkable duality with
 link invariants in the Chern-Simons theory approach of Witten \cite{Wit1, Wit2}.
A mathematical theory of the topological vertex  has been developed in  \cite{LLLZ}.

The topological vertex, which is the generating function of the Gromov-Witten invariants of $\mathbb{C}^3$ with three
special $D$-branes,
is a mysterious combinatorial object that asks for further studies.
On the $A$-theory side,
the topological vertex can be
realized as a state in the threefold tensor product of the space $\Lambda$ of symmetric functions.
In this representation its expressions given
by physicists \cite{AKMV} or by mathematicians \cite{LLLZ} are both very complicated.
It is very interesting to understand the topological vertex
from other perspectives.
In \cite{ORV},
the topological vertex is related to a combinatorial problem
of plane partitions.
In \cite{AKMV} it was suggested that the topological vertex is a Bougoliubov
transform via the boson-fermion correspondence.
This point of view was further elaborated in \cite{ADKMV} and extended
to the partition functions of toric Calabi-Yau threefolds.
Indeed, by the local mirror symmetry \cite{HV, HIV},
on the B-model side,
one studies quantum Kodaira-Spencer theory of the local mirror curve.
By physical derivations,
the corresponding state is constrained by the Ward identities,
giving the $W_\infty$ constraints.
(See also \cite{Gukov-Sulkowski} where the partition functions
are expected to be annihilated by certain quantum operators
obtaining by quantized the local mirror curves.)
In this formalism it is natural to use the fermonic picture,
and a simple looking formula for the fermionic form of the topological vertex
under the boson-fermion correspondence was conjectured
in \cite{ADKMV}.
The ADKMV conjecture is directly related to integrable hierarchies:
The one-legged case is related to the KP hierarchy, the two-legged case
to the 2-dimensional Toda hierarchy, and the three-legged case to the 3-component KP hierarchy (see Remark \ref{rem:Bogoliubov-KP}).
The one-legged and the two-legged cases can also be seen directly from the bosonic picture \cite{Zh1},
but the three-legged case can only be seen through the fermionic picture.

The topological vertex can be used to compute Gromov-Witten invariants
of toric Calabi-Yau $3$-folds by certain gluing rules.
There is a standard inner product on the space $\Lambda$,
and the gluing rule is essentially taking inner product over the components corresponding to
the branes of gluing (see \S\S \ref{subsec:GV of Resolved Boson} for exact formulation).
So the resulted generating functions are states in multifold tensor products of the space $\Lambda$.
In general, they
have very complicated combinatorial structures.

In our recent work \cite{DZ}, we proposed a generalization of the  ADKMV conjecture to the framed topological vertex which we refer to as
the framed ADKMV conjecture.
Note that it is important to consider framing
when we consider gluing of the topological vertex. We gave a proof in \cite{DZ} of the framed ADKMV conjecture in the one-legged case and the two-legged case, and derived
a determinantal formula for the framed topological vertex in the three-legged case based on the Framed ADKMV Conjecture.
It remains open to give a proof of this conjecture for the full three-legged topological vertex.

Provided that the framed ADKMV conjecture holds,
then a  natural question is whether or not the generating functions of the Gromov-Witten invariants of general toric Calabi-Yau threefolds are  Bogoliubov transforms
in the fermionic picture.
It was also conjectured in \cite{ADKMV} that it is indeed the case.
However, it seems very difficult to prove this conjecture directly by boson-fermion correspondence and standard Schur calculus,
even for the very simple case of the resolved conifold with a single brane.
In \cite{Sulkowski}
the closed string partition function of the resolved conifold
is related to Hall-Littlewood functions
and a fermionic represenation is obtained by
the deformed boson-fermion correspondence.
Based on the method in \cite{ADKMV},
it was shown in \cite{Kashani-Poor} that the B-model amplitude  of the mirror space
of the one-legged resolved conifold is a Bogoliubov transform.
It also seems difficult to generalize the method in \cite{ADKMV}
and \cite{Kashani-Poor} to prove this conjecture in general.

In this paper we will tackle this problem using a different strategy.
We will  start from the framed ADKMV conjecture,
 and then consider the gluing rule of the topological vertex
 as presented in \cite{AKMV}\cite{LLLZ} in the fermionic picture.
In this work we will focus on the framed one-legged resolved conifold.
 But the method here can be easily modified to the cases of the total spaces of
 vector bundles $O(p)\oplus O(-2-p)\rightarrow \mathbb{P}^1$, $p\in\mathbb{Z}$.
 The treatment for general toric Calabi-Yau threefolds
will be presented in a separate paper \cite{DZ'}.
The main result of the present paper is that the generating function of the
 Gromov-Witten invariants of the resolved conifold with one brane and arbitrary framing
 is a Bogoliubov transform of the fermionic vacuum;
 in particular, it is a tau function of the KP hierarchy.

The rest of the paper is arranged as follows. After reviewing some preliminaries and fixing notations in \S 2, we rewrite in \S 3 the generating function of Gromov-Witten
invariants of the framed one-legged resolved conifold as an gluing (see that section  for precise meaning) of two fermionic states which are Bogoliubov transforms, based on the fermionic representation of the framed one-legged and two-legged topological vertex. In the final \S 4, we prove that the gluing of an arbitrary two-component Bogoliubov transform and an arbitrary one-component Bogoliubov transform is also a Bogoliubov transform. The result in \S 4, combing with \S 3, leads directly to the result that the fermionic representation of the generating function considered is a Bogoliubov transform.

\vspace{.1in}
{\em Acknowledgements}.
The work was partially done during the first author's attending 
the mathematical seminars supported by Morningside Center of CAS.
The first author is partially supported by NSFC grants
(11001148 and 10901152) and the President Fund of GUCAS. The second author is partially supported by two NSFC grants (10425101 and 10631050)
and a 973 project grant NKBRPC (2006cB805905).

\section{Preliminaries}\label{sec:prelimilaries}

In this section, we recall briefly some well-known concepts and results that will be used in following sections.

\subsection{Partitions}
A partition $\mu$ of a positive integral number $n$ is a decreasing finite sequence of integers $\mu_1\geq\cdots \geq\mu_l>0$,
such that $|\mu| = \mu_1 + \cdots + \mu_l = n$.
The following number associated to $\mu$ will be useful in this paper:
\be
 \kappa_\mu = \sum_{i=1}^l \mu_i(\mu_i - 2i + 1).
\ee
It is very useful to graphically represent a partition by its Young diagram.
This leads to many natural definitions.
First of all,
by transposing the Young diagram one can define the conjugate $\mu^t$ of $\mu$.
Secondly
assume the Young diagram of $\mu$ has $k$ boxes in the diagonal.
Define $m_i = \mu_i - i$ and $n_i = \mu^t_i - i$ for $i = 1, \cdots , k$,
then it is clear that $m_1> \cdots > m_k \geq 0$ and  $n_1> \cdots > n_k \geq 0$.
The partition $\mu$ is completely determined by the numbers $m_i , n_i$.
We often denote the partition $\mu$ by $(m_1, \dots , m_k | n_1, \dots , n_k)$,
this is called the Frobenius notation.
A partition of the form $(m|n)$ in Frobenius form is called a hook partition.

For a box $e$ at the position $(i , j)$ in the Young diagram of $\mu$,
define its content by $c(e) = j-i$.
Then it is easy to see that
\be \label{eqn:Kappa}
\kappa_\mu = 2\sum_{e\in \mu}c(e).
\ee
Indeed,
\ben
&& \sum_{e\in \mu}c(e) = \sum_{i=1}^l \sum_{j=1}^{\mu_i} (j-i)
= \sum_{i=1}^n (\half \mu_i(\mu_i+1) - i \mu_i) = \half \kappa_\mu.
\een
A straightforward application of \eqref{eqn:Kappa} is the following:

\begin{lem}[e.g see \cite{DZ}] \label{lm:kappa}
Let $\mu = (m_1, m_2, \dots, m_k | n_1, n_2, \dots, n_k)$ be a partition written in the Frobenius notation. Then we have
\be
\kappa_\mu = \sum_{i=1}^k m_i(m_i+1) - \sum_{i=1}^k n_i(n_i+1).
\ee
In particular,
\be
\kappa_{(m_1, m_2, \dots, m_k | n_1, n_2, \dots, n_k)}
= \sum_{i=1}^k \kappa_{(m_i|n_i)}.
\ee
\end{lem}

\subsection{Schur functions and skew Schur functions}\label{subsec:Schur & skew schur}
Let $\Lambda$ be the space of symmetric functions in $\bx = (x_1, x_2, \dots)$.
The inner product on the space $\Lambda$ is defined by setting the set of Schur functions as an orthonormal basis.
For a partition $\mu$, let $s_\mu:=s_\mu(\bx)$ be the corresponding Schur function in $\Lambda$. Given to partitions $\mu$ and $\nu$,
 the skew Schur functions $s_{\mu/\nu}$ is defined by the condition \cite{Macdonald}
$$(s_{\mu/\nu} , s_\lambda) = (s_\mu , s_\nu s_\lambda)$$
for all partitions $\lambda$. This is equivalent to define
$$s_{\mu/\nu} = \sum_{\lambda}c_{\nu\lambda}^\mu s_\lambda,$$
where the constants $c_{\nu\lambda}^\mu$ are the structure constants
(called the Littlewood-Richardson coefficients) defined by
\be
s_\nu s_\lambda = \sum_{\gamma}c_{\nu\lambda}^\gamma s_\gamma.
\ee\\

We often meet some specialization of symmetric functions.
Let $q^{\rho}:=(q^{-1/2}, q^{-3/2}, \dots)$.
It is easy to see that
\be
p_n(q^\rho) = \frac{1}{q^{n/2} - q^{-n/2}} = \frac{1}{[n]},
\ee
where $[n] = q^{n/2} - q^{-n/2}.$ A very interesting fact is that with this specialization the Schur functions
also have very simple expressions.

\begin{prop} \label{prop:SchurSpec} \cite{Zh4}   For any partition $\mu$, one has
$$s_\mu(q^\rho) = q^{\kappa_\mu/4}\frac{1}{\prod_{e\in \mu}[h(e)]},$$
where $h(e)$ is the hook number of $e$ and $[n] = q^{n/2} - q^{-n/2}.$
\end{prop}

\subsection{Fermionic Fock space }
We say a set of integers $A = \{a_1, a_2, \dots \}\subset \mathbb{Z}+\frac{1}{2}$, $a_1>a_2> \cdots$, is admissible if it satisfies the following two conditions:
\begin{itemize}
\item[1.] $\mathbb{Z}_- + \frac{1}{2}\backslash A$ is finite and
\item[2.] $A\backslash \mathbb{Z}_- + \frac{1}{2}$ is finite,
\end{itemize}
where  $\mathbb{Z}_-$ is the set of negative integers.

Consider the linear space $W$ spanned by a basis $\{\underline{a}| a\in \mathbb{Z}+\frac{1}{2}\}$,
indexed by half-integers.
For an admissible set $A = \{a_1, a_2, \dots\}$,
we associate an element $\underline{A}\in \wedge^\infty W$ as follows:
$$\underline{A} = \underline{a_1}\wedge \underline{a_2} \wedge \cdots.$$
Then the free fermionic Fock space $\mathcal{F}$ is defined as
$$\cF = \Span \{\underline{A}: \; A\subset \mathbb{Z}+\frac{1}{2}\; \text{is admissible} \}.$$
One can define an inner product on $\mathcal{F}$ by taking
$\{\underline{A}:\; A\subset \mathbb{Z}+\frac{1}{2}\; \text{is admissible} \}$ as an orthonormal basis.

For $\underline{A} = \underline{a_1}\wedge \underline{a_2} \wedge \cdots
\in \mathcal{F}$,
define its charge as:
$$|A\backslash \mathbb{Z}_- + \frac{1}{2}| - |\mathbb{Z}_- + \frac{1}{2}\backslash A|.$$
Denote by $\cF^{(n)} \subset \mathcal{F}$ the subspace spanned by $\underline{A}$ of charge $n$,
then there is a decomposition
$$\mathcal{F} = \bigoplus_{n\in \mathbb{Z}} \cF^{(n)}.$$
An operator on $\mathcal{F}$ is called charge 0 if it preserves the above decomposition.

The charge 0 subspace $F^{(0)}$ has a basis indexed by partitions:
\be
|\mu\rangle:= \underline{\mu_1 - \frac{1}{2}} \wedge \underline{\mu_2 - \frac{3}{2}}\wedge \cdots \wedge
 \underline{\mu_l -\frac{2l-1}{2}}\wedge \underline{-\frac{2l+1}{2}}\wedge \cdots
\ee
where $\mu = (\mu_1, \cdots , \mu_l)$,
i.e.,
$|\mu\rangle = \underline{A_\mu}$, where $A_\mu =(\mu_i - i + \half)_{i=1, 2, \dots}$.
If $\mu = (m_1, \cdots , m_k | n_1, \cdots , n_k)$ in Frobenius notation, then
\begin{equation}
|\mu\rangle = \underline{m_1+\frac{1}{2}}\wedge \cdots \wedge \underline{m_k+\frac{1}{2}}\wedge \underline{-\frac{1}{2}}
\wedge \underline{-\frac{3}{2}} \wedge \cdots \wedge \widehat{\underline{-n_k-\frac{1}{2}}} \wedge \cdots \wedge
\widehat{\underline{-n_1-\frac{1}{2}}} \wedge \cdots .
\end{equation}
In particular,
when $\mu$ is the empty partition,
we get:
$$|0\rangle := \underline{-\frac{1}{2}}\wedge \underline{-\frac{3}{2}}\wedge \cdots \in \mathcal{F}.$$
It will be called the fermionic vacuum vector.

We now recall the creators and annihilators on $\mathcal{F}$.
For $r \in \mathbb{Z}+\frac{1}{2}$,
define operators $\psi_r$ and $\psi^*_r$ by
\begin{eqnarray*}
&\psi_r (\underline{A}) =
\begin{cases}
(-1)^{k}\underline{a_1}\wedge\cdots\wedge\underline{a_k}\wedge \underline{r}\wedge\underline{a_{k+1}}\wedge\cdots, & \text{if $a_k > r > a_{k+1}$ for some $k$}, \\
0, &  \text{otherwise};
\end{cases}\\
&\psi^*_r(\underline{A}) =
\begin{cases}
(-1)^{k+1}\underline{a_1}\wedge\cdots\wedge \widehat{\underline{a_k}}\wedge\cdots, & \text{if $a_k = r$ for some $k$}, \\
0, &  \text{otherwise}.
\end{cases}
\end{eqnarray*}
Under the inner product defined above, for $r \in \mathbb{Z}+1/2$, it is clear that $\psi_r $ and $\psi^*_r$ are adjoint operators.
The anti-commutation relations for these operators are
\begin{equation} \label{eqn:CR}
[\psi_r,\psi^*_s]_+:= \psi_r\psi^*_s + \psi^*_s\psi_r = \delta_{r,s}id
\end{equation}
and other anti-commutation relations are zero.
It is clear that for $r > 0$,
\begin{align}
\psi_{-r} \vac & = 0, & \psi_r^* \vac & = 0,
\end{align}
so the operators $\{\psi_{-r}, \psi_r^*\}_{r > 0}$ are called the fermionic annihilators.
For a partition $\mu = (m_1, m_2, . . ., m_k | n_1, n_2, . . ., n_k)$, it is clear that
\be\label{eq:operator rep for mu}
|\mu\rangle = (-1)^{n_1 + n_2 + . . . + n_k}\prod_{i=1}^k \psi_{m_i+\frac{1}{2}} \psi_{-n_i-\frac{1}{2}}^*|0\rangle .
\ee
So the operators $\{\psi_{r}, \psi_{-r}^*\}_{r > 0}$ are called the fermionic creators.
The normally ordered product is defined as
\begin{equation*}
:\psi_r\psi^*_r: =
\begin{cases}
 \psi_r\psi^*_r, & r>0, \\
- \psi^*_r\psi_r, & r<0.
\end{cases}
\end{equation*}
In other words,
an annihilator is always put on the right of a creator.

\subsection{Boson-fermion correspondence}
For any integer $n$, define an operator $\alpha_n$ on the fermionic Fock space $\mathcal{F}$ as follows:
\begin{equation*}
\alpha_n = \sum_{r\in \mathbb{Z} + \frac{1}{2}}:\psi_r\psi^*_{r+n}:
\end{equation*}
Let
$\mathcal{B} = \Lambda[z , z^{-1}]$
be the bosonic Fock space, where $z$ is a formal variable.
Then the  boson-fermion correspondence is a linear isomorphism
$\Phi: \mathcal{F} \rightarrow \mathcal{B}$ given by
\begin{equation}
u\mapsto z^m \langle\underline{0}_m | e^{\sum_{n=1}^\infty \frac{p_n}{n}\alpha_n}u\rangle ,\ \ u\in \cF^{(m)}
\end{equation}
where $|\underline{0}_m\rangle = \underline{-\frac{1}{2}+m}\wedge \underline{-\frac{3}{2}+m}\wedge\cdots$.
It is clear that $\Phi$ induces an isomorphism between $\cF^{(0)}$ and $\Lambda$. Explicitly, this isomorphism is given by
\begin{equation}\label{boson-fermion}
|\mu\rangle \longleftrightarrow s_\mu.
\end{equation}

The boson-fermionic correspondence plays an important role in Kyoto school's theory
of integrable hierarchies.
For example,

\begin{prop}\label{bilinear relation tau fermion}
If $\tau\in \Lambda$ corresponds to $|v\rangle\in F^{(0)}$, then $\tau$ is a $tau$-function of the KP
hierarchy in the Miwa variable $t_n = \frac{p_n}{n}$ if and only if $|v\rangle$ satisfies the bilinear relation
\begin{equation}\label{bilinear relation tau fermion 1}
\sum_{r\in \mathbb{Z} + \frac{1}{2}}\psi_r |v\rangle\otimes \psi^*_r |v\rangle = 0.
\end{equation}
\end{prop}

\begin{rem}\label{rem:Bogoliubov-KP}
A state $|v\rangle\in \cF^{(0)}$ satisfies the bilinear relation
\eqref{bilinear relation tau fermion 1} if and only if it lies in the orbit $\widehat{GL_\infty}|0\rangle$.
This is equivalent to say that $|v\rangle$ can be represented as
$$|v\rangle = \exp(\sum_{r,s\in \mathbb{Z}+1/2}M_{rs}:\psi_{r}\psi_{s}^*:)|0\rangle$$
for some coefficients $M_{rs}$.
There is also a multi-component generalization of the boson-fermion correspondence
which can be used to study multi-component KP hierarchies \cite{KL}.
\end{rem}

\section{Gromov-Witten invariants of the framed one-legged resolved conifold and its fermionic form}
\subsection{Generating function of Gromov-Witten invariants of the framed one-legged resolved conifold}\label{subsec:GV of Resolved Boson}
By the theory of topological vertex \cite{AKMV}\cite{LLLZ}, the Gromov-Witten invariants of any toric Calabi-Yau threefold
can be computed from the topological vertex by certain explicit gluing process. As a special case, the Gromov-Witten invariants of
the framed one-legged resolved conifold can be computed by gluing the framed one-legged topological vertex and the framed two-legged topological vertex.

The one-legged topological vertex with framing $a$, in terms of Schur functions, is given by
\be
Z_1^{(a)}(\by) =\sum_{\mu} q^{a\kappa_\mu/2}s_\mu(q^\rho)s_\mu(\by)
\ee
and the two-legged topological vertex  with framings $(a_1 , a_2)$, in terms of skew Schur functions, is given by\cite{Zh1} \cite{Zh4}
\be
Z^{(a_1, a_2)}_2(\bx; \by)
= \sum_{\mu, \nu}\left(q^{\frac{(a_1+1)\kappa_\mu+a_2\kappa_\nu}{2}}
\sum_{\eta}s_{\mu^t/\eta}(q^\rho)s_{\nu/\eta}(q^\rho)\right )s_\mu(\bx) s_\nu(\by),
\ee
where $q=e^{-g_s}$ and $g_s$ the coupling constant , $\bx=(x_1,x_2,\cdots)$ and $\by=(y_1, y_2, \cdots)$, and the partitions $\mu$, $\nu$ encodes boundary conditions of the holomorphic curves we considered in $\mathbb{C}^3$.

By the theory of the topological vertex,  the generating function  of the Gromov-Witten
 invariants of the  resolved conifold with one brane of framing $a$ is given by
\be\label{eq: boson gluing}
\tilde{Z^a}(\bx) = \sum_{\mu}\left(\sum_{\nu}C^2_{\mu\nu}(a)Q^{|\nu|}C^1_{\nu^t}\right)s_\mu(\bx)
\ee
where $Q = - e^{-t}$ and $t$ is the K\"{a}hler parameter of the $\mathbb{P}^1$  in the resolved conifold , and
\be
\begin{split}
&C^2_{\mu\nu}(a) = q^{\frac{(a+1)\kappa_\mu}{2}}\sum_{\eta}s_{\mu^t/\eta}(q^\rho)s_{\nu/\eta}(q^\rho),\\
&C^1_\mu = s_\mu(q^\rho).
\end{split}
\ee

Let $Z_0 = \sum_{\mu}s_\mu(q^\rho)Q^{|\mu|}s_{\mu^t}(q^\rho)$, it is the generating function of the closed Gromov-Witten invariants of the resolved conifold. We are interested in  the normalized generating function $Z^a(\bx) = \tilde{Z^a}(\bx)/Z_0$, which is the generating function of the open Gromov-Witten invariants of the famed one-legged resolved conifold.

From now on, we will view $Q$ as a formal variable, $Z_0$ as a formal power series of $Q$, and $\tilde{Z}^a(\bx)$ and $Z^a(\bx)$ as formal power series of $Q$ with coefficients in the space of symmetric functions (with parameter $q$).

\subsection{Fermionic representation of the generating function}
According to the boson-fermion correspondence \eqref{boson-fermion}, the element $V$ in the fermionic Fock space $\mathcal{F}$
corresponding to the normalized generating function $Z^a(\bx)$ defined in the previous subsection is
\be
V = \sum_{\mu}\left(\sum_{\nu}C^2_{\mu\nu}(a)Q^{|\nu|}C^1_{\nu^t}/Z_0\right)|\mu\rangle
\ee\\

For simplicity, for an integer $m > 0$, we denote $m+1/2$ by $\bm$ and $-m-1/2$ by $-\bm$. The main aim  of the present paper is to prove the following
\begin{thm}\label{thm:fernion rep}
The element $V$ defined as above is a Bogoliubov transform of the vacuum in $\mathcal{F}$. In other word,
for $m , n \geq 0$, there exist certain coefficients $R_{mn}$ as formal power series of $Q$, such that
\be\label{eq:fermion rep}
V = \exp(\sum_{m , n \geq 0} R_{mn}\psi_\bm\psi^{*}_{-\bn})|0\rangle
\ee
where $|0\rangle$ is the vacuum vector in $\mathcal{F}$. In particular, $Z^a(\bx)$ is a tau function of the KP hierarchy in the Miwa variables $t_n=\frac{p_n(\bx)}{n}$.
\end{thm}

It seems very difficult to prove Theorem \ref{thm:fernion rep} by standard Schur calculus. The starting point of the proof here is the formulas
of the fermionic representation of the framed one-legged and two-legged topological vertex which were conjectured in \cite{ADKMV} and proved in
our recent work \cite{DZ}.

Under the boson-fermion correspondence \eqref{boson-fermion}, the element $V^{(a)}_1 \in \mathcal{F}$ corresponding to the framed one-legged
topological vertex $Z_1^a(\by)$ is
\be
V^{(a)}_1 = \sum_{\mu} q^{a\kappa_\mu/2}s_\mu(q^\rho)|\mu\rangle
\ee
and the element $V_2^{(a_1 , a_2)} \in \mathcal{F}_1\otimes \mathcal{F}_2 $ corresponding to the framed two-legged topological vertex
$Z^{(a_1, a_2)}_2(\bx; \by)$ is
\be
V^{(a_1 , a_2)}_2 = \sum_{\mu, \nu}\left(q^{\frac{(a_1+1)\kappa_\mu+a_2\kappa_\nu}{2}}
\sum_{\eta}s_{\mu^t/\eta}(q^\rho)s_{\nu/\eta}(q^\rho)\right )|\mu\rangle\otimes|\nu\rangle
\ee
where $\mathcal{F}_1$ and $\mathcal{F}_2$ are two copies of $\mathcal{F}$.

On the two-component femionic Fock space $\mathcal{F}_1 \otimes \mathcal{F}_2$,
define for $i=1, 2$ operators $\psi^i_r$ and $\psi^{i*}_r$, $r \in {\mathbb Z} + \half$.
They act on the $i$-th factor of the tensor product as the operators $\psi_r$ and$\psi_r^*$ respectively,
and we use the Koszul sign convention for the anti-commutation relations of these operators, i.e., we set
\be\label{eqn:sign convention}[\psi^i_r , \psi^j_s]_+=[\psi^i_r ,\psi^{j*}_s]_+=[\psi^{i*}_r , \psi^{j*}_s]_+ =0\ee
for $i\neq j$ and $r , s \in {\mathbb Z} + \half$.

Note that the charge 0 subspace $(\mathcal{F}_1\otimes \mathcal{F}_2)^{(0)}$ of $\mathcal{F}_1\otimes \mathcal{F}_2$ has a natural
decomposition  as
\be
(\mathcal{F}_1\otimes \mathcal{F}_2)^{(0)} = \bigoplus_{n\in \mathbb{Z}}(\mathcal{F}_1^{(n)}\otimes \mathcal{F}_2^{(-n)})
\ee
For an element $W \in (\mathcal{F}_1\otimes \mathcal{F}_2)^{(0)}$, we denote  by $W^0$ the projection of $W$ to the component
 $\mathcal{F}_1^{(0)}\otimes \mathcal{F}_2^{(0)}$ with respect to this decomposition.

The main result proved in \cite{DZ} is the following
\begin{thm}\label{thm:fermion rep. of top. vertex}
For $m , n \geq 0$ and $i , j = 1 , 2$, there exist coefficients $A_{mn}=A_{mn}(a)$ and $A_{mn}^{ij}=A^{ij}_{mn}(a_1 ,a_2)$ such that
\be
\begin{split}
&V^{(a)}_1 = \exp(\sum_{m , n \geq 0} A_{mn}(a)\psi_\bm\psi^{*}_{-\bn})|0\rangle\\
&V^{(a_1 , a_2)}_2 = \left(\exp(\sum_{i,j=1,2}\sum_{m , n \geq 0}A^{ij}_{mn}(a_1 ,a_2)\psi^i_\bm\psi^{j*}_{-\bn})|0_{12}\rangle\huge\right)^0
\end{split}
\ee
where $|0_{12}\rangle = |0_{1}\rangle\otimes|0_{2}\rangle$ is the vacuum vector in $\mathcal{F}_1\otimes \mathcal{F}_2 $.
\end{thm}

The coefficients  $A_{mn}^{ij}$ and $A_{mn}$ in Theorem \ref{thm:fermion rep. of top. vertex} can be given explicitly (see \cite{ADKMV}\cite{DZ}):
\begin{equation*}
\begin{split}
&A_{mn}(a) = (-1)^n q^{(2a+1)(m(m+1)-n(n+1))/4}\frac{1}{[m+n+1][m]![n]!}, \\
&A^{11}_{mn}(a_1 ,a_2)=(-1)^n q^{\frac{(2a_1+1)m(m+1)-(2a_{2}+1)n(n+1)}{4}+\frac{1}{6}}\sum_{l=0}^{\min(m , n)}
 \frac{q^{\frac{1}{2}(l+1)(m+n-l)}}{[m-l]![n-l]!}, \\
&A^{22}_{mn}(a_1 ,a_2)=(-1)^n q^{\frac{(2a_2+1)m(m+1)-(2a_{1}+1)n(n+1)}{4}+\frac{1}{6}}\sum_{l=0}^{\min(m , n)}
 \frac{q^{\frac{1}{2}(l+1)(m+n-l)}}{[m-l]![n-l]!}, \\
&A^{12}_{mn}(a_1 ,a_2)=(-1)^n q^{\frac{(2a_1+1)m(m+1)-(2a_{2}+1)n(n+1)}{4}+\frac{1}{6}}\sum_{l=0}^{\min(m , n)}
 \frac{q^{\frac{1}{2}(l+1)(m+n-l)}}{[m-l]![n-l]!},\\
&A^{21}_{mn}(a_1 ,a_2)=-(-1)^n q^{\frac{(2a_2+1)m(m+1)-(2a_{1}+1)n(n+1)}{4}-\frac{1}{6}}\sum_{l=0}^{\min(m , n)}
 \frac{q^{-\frac{1}{2}(l+1)(m+n-l)}}{[m-l]![n-l]!}.
\end{split}
\end{equation*}

Let $\tilde{V} = Z_0 V \in \mathcal{F}$, define
\be V'_1 = \sum_\mu |Q|^{|\mu|}s_{\mu^t}(q^\rho)|\mu\rangle
\ee
 as a formal power series of $Q$
with coefficients in $\mathcal{F}$. We view $\tilde{V}$ as an element in  $\mathcal{F}_1$ and $V'_1$ an
element in $\mathcal{F}_2$. By the definition of the inner product on $\mathcal{F}$ and \eqref{eq: boson gluing}, it is clear that
\be\label{eq:inner product rep}
\tilde{V} = (V^{(a ,0)}_2 , V'_1)
\ee
where the inner product is taken on the $\mathcal{F}_2$ component.

The following lemma, which shows that $V'_1$ is also a Bogoliubov transform of the fermionic vacuum, will be used in our proof of
Theorem \ref{thm:fernion rep}.

\begin{lem}
The element $V'_1$ is a Bogoliubov transform of the fermionic vacuum, i.e, it can be represented as
\be
V'_1 = \exp(\sum_{m , n \geq 0} A'_{mn}\psi_\bm\psi^{*}_{-\bn})|0\rangle,
\ee
where the coefficients
\be
A'_{mn} = Q^{m+n+1}q^{\frac{1}{2}(n(n+1) - m(m+1))}A_{mn}(0).
\ee
\begin{proof}\label{lem:fermion for V'}
Let $\mu = (m_1, \cdots , m_k | n_1 , \cdots , n_k)$ be a partition in Frobenius notation.
Then $\kappa_{\mu^t} = -\kappa_\mu$ by Lemma \ref{lm:kappa}, and hence $s_{\mu^t}(q^\rho) = q^{-\frac{\kappa_\mu}{2}}s_{\mu}(q^\rho)$
by Proposition \ref{prop:SchurSpec}.
By \eqref{eqn:CR}, \eqref{eq:operator rep for mu} and Theorem \ref{thm:fermion rep. of top. vertex},
one can show that $s_{\mu}(q^\rho) = det(A_{m_in_j}(0))_{k\times k}$. Note that $|\mu| = \sum_{i=1}^k (m_i+n_i+1)$. If we take
$A'_{mn} = Q^{m+n+1}q^{\frac{1}{2}(n(n+1)-m(m+1))}A_{mn}(0)$, then we have $Q^{|\mu|}s_{\mu^t}(q^\rho) = det(A'_{m_in_j})_{k\times k}$, and hence
$V'_1 = \exp(\sum_{m , n \geq 0} A'_{mn}\psi_\bm\psi^{*}_{-\bn})|0\rangle$.
\end{proof}
\end{lem}

By \eqref{eq:inner product rep},
Theorem \ref{thm:fermion rep. of top. vertex} and Lemma \ref{lem:fermion for V'},
it is clear that Theorem \ref{thm:fernion rep} is a direct corollary
of Theorem \ref{thm:gluing} that we will prove in \S \ref{sec:special fermion gluing}.

Provided Theorem \ref{thm:fernion rep}, it is easy to determine the coefficients $R_{mn}$ appearing in \eqref{eq:fermion rep}. They are given by
\be
R_{mn} = \sum_{\nu}C^2_{(m|n)\nu}(a)Q^{|\nu|}s_{\nu^t}(q^\rho)/Z_0.
\ee
But it is not our aim here to study various simple forms of $R_{mn}$.

\section{Gluing of  Bogoliubov transforms}\label{sec:special fermion gluing}
We have mentioned the notion of (one-component) Bogoliubov transform in the previous section.
Recall that a vector $V$ in the 2-fold fermionic Fock space $\cF_1\otimes\cF_2$ is called a Bogoliubov transform (of the fermionic vacumm) if it is given by the vacuum in $\cF_1\otimes\cdots\otimes \cF_n$  acted upon by an exponential of a quadratic expression of fermionic creators. In other word, it can be represented as
\be
V = \exp(\sum_{i,j = 1}^2\sum_{m , n \geq 0}A^{ij}_{mn}\psi^i_\bm\psi^{j*}_{-\bn})|0\rangle
\ee
where $|0\rangle$ is the vacuum in $\cF_1\otimes \cF_2$ and $A^{ij}_{mn}$ are certain coefficients possibly with parameters.

In this section, we study properties of  fermionic states which are constructed by gluing (to be defined later) one-component and two-component Bogoliubov transforms. If we glue an arbitrary one-component Bogoliubov transform and an arbitrary two-component
Bogoliubov transform, then we get a state in the fermionic Fock space $\cF$. Our aim here is to prove that
the state we get is also a Bogoliubov transform of the fermionic vacuum. In particular, it is a tau function of the KP hierarchy.

Let
\begin{equation}
V_1 = \exp(\sum_{i,j = 1,2}\sum_{m , n \geq 0}A^{ij}_{mn}\psi^i_\bm\psi^{j*}_{-\bn})|0_{12}\rangle
\end{equation}
be a two-component Bogoliubov transform in $\mathcal{F}_1\otimes \mathcal{F}_2$ and
\begin{equation}
V_2 = \exp(\sum_{m , n \geq 0}A_{mn}\psi^2_\bm\psi^{2*}_{-\bn})|0_{2}\rangle
\end{equation}
be an one-component Bogoliubov transform in $\mathcal{F}_2$, where $A^{ij}_{mn}$ and $A_{mn}$ are arbitrary coefficients maybe with parameters.

Define
\begin{equation}
\tilde{V_2} = \exp(\sum_{m , n \geq 0}Q^{m+n+1}A_{mn}\psi^2_\bm\psi^{2*}_{-\bn})|0_{2}\rangle
\end{equation}
where $Q$ is a formal variable and $\tilde{V_2}$ is viewed as a formal power series of $Q$ with coefficients in $\mathcal{F}_2$.

 The following inner product
\begin{equation}
\begin{split}
V_0 &= \left(\exp(\sum_{m , n \geq 0}A^{22}_{mn}\psi^2_\bm\psi^{2*}_{-\bn})|0_{2}\rangle ,
 \exp(\sum_{m , n \geq 0}Q^{m+n+1}A_{mn}\psi^2_\bm\psi^{2*}_{-\bn})|0_{2}\rangle\right)\\
    &=\langle 0_2|\exp(\sum_{m , n \geq 0}Q^{m+n+1}A_{mn}\psi^2_{-\bn}\psi^{2*}_{\bm})
     \exp(\sum_{m , n \geq 0}A_{mn}^{22}\psi^2_\bm\psi^{2*}_{-\bn})|0_{2}\rangle
\end{split}
\end{equation}
is well defined as a formal power series of the formal variable $Q$.

Define $\tilde{V} = (V_1 , \tilde{V_2})$,
where the inner product is taken on the $\mathcal{F}_2$ component. Then $\tilde{V}$ is  a formal power series of $Q$ with coefficients in $\mathcal{F}_1$.

\begin{thm}\label{thm:gluing}
Let $\tilde{V}$ and $V_0$ be as above,
then the formal  power series $V = \tilde{V}/V_0$ of $Q$ with coefficients in $\mathcal{F}_1$
is a Bogoliubov transform of the fermionic vacuum $|0_1\rangle \in \mathcal{F}_1$,
 i.e., for  $m , n \geq 0$, there exist  formal power series $R_{mn}$ of $Q$, such that
 \be
 V = \exp(\sum_{m,n\geq 0}R_{mn}\psi^1_\bm\psi^{1*}_{-\bn})|0_{1}\rangle.
 \ee
\end{thm}
From the proof of this theorem, we will see why $V_0$ appear naturally as a factor of $\tilde{V}$.
The following lemma, which is well-known in Lie theory,
will be used in the proof of Theorem \ref{thm:gluing}.

\begin{lem} \label{lem:exponential cr}
(See e.g \cite{Hall}) Let $A$ and $B$ be two linear operators on a
vector space $H$. Assume both $e^A$ and $e^B$ make sense. If the
commutator $[A , B] = AB - BA$ commutes with both $A$ and $B$. Then
\be e^A e^B = e^{[A , B]}e^B e^A. \ee
\end{lem}

\begin{proof}
(Proof of Theorem \ref{thm:gluing}) Recall that
\begin{equation*}
\tilde{V}=(\exp(\sum_{i,j = 1,2}\sum_{m , n \geq 0}A^{ij}_{mn}\psi^i_\bm\psi^{j*}_{-\bn})|0_{12}\rangle, \exp(\sum_{m , n \geq 0}Q^{m+n+1}A_{mn}\psi^2_\bm\psi^{2*}_{-\bn})|0_{2}\rangle).
\end{equation*}
To make the notations simpler,
let
\bea
&& \cA^{ij} = \sum_{m , n \geq 0}A^{ij}_{mn}\psi^i_{\bm}\psi^{j*}_{-\bn}, \\
&& \cA = \sum_{m , n \geq 0}Q^{m+n+1}A_{mn} \psi^2_{\bm}\psi^{2*}_{-\bn}.
\eea
Then one can rewrite $\tilde{V}$ as follows:
\be
\tilde{V} = \exp{\cA^{11}} \langle0_2|\exp(\cA^*) \exp (\cA^{21}) \exp (\cA^{12}) \exp (\cA^{22})
|0_2\rangle |0_1\rangle
\ee
As usual,
our strategy here is to move the annihilators to the right using the
anticommutation relations \eqref{eqn:CR}.
By \eqref{eqn:CR} and \eqref{eqn:sign convention}, one has
\be
[\cA^*, \cA^{21}]
%% = \left[\sum_{m , n \geq 0}Q^{m+n+1}A_{mn}\psi^2_{-\bn}\psi^{2*}_{\bm},
%% \sum_{m , n \geq 0}A^{21}_{mn}\psi^{2}_{\bm}\psi^{1*}_{-\bn} ,  \right]\\
= \cB^{21}=\sum_{m, n \geq 0} B^{21}_{mn} \psi^{2}_{-\bm}\psi^{1*}_{-\bn},
\ee
where
\be
B^{21}_{mn} = \sum_{r \geq 0}Q^{r+m+1}A_{rm}A^{21}_{rn}
\ee
are formal power series of $Q$ which are divisible by $Q$.
Note that the right-hand side  commutes with both
$\cA^*$ and
$\cA^{21}$,
hence by Lemma \ref{lem:exponential cr} we get
\ben
&& \exp(\cA^*) \exp(\cA^{21})
= \exp(\cA^{21}) \exp(\cA^*) \exp(\cB^{21}).
\een
By the same method, one can show that
\be
\exp(\cB^{21})\exp(\cA^{12})
=  \exp(\cA^{2112,1}) \exp(\cA^{12}) \exp(\cB^{21}),
\ee
where
\be
\cA^{2112,1} = [\cB^{21}, \cA^{12}]
= - \sum_{m,n \geq 0} \sum_{r\geq 0} A^{12}_{mr}B^{21}_{rn}\psi^1_m\psi^{1*}_{-n},
\ee
and
\be
\exp(\cB^{21})\exp(\cA^{22})
=  \exp(\cA^{21,1}) \exp(\cA^{22}) \exp(\cB^{21}),
\ee
where
\be
\cA^{21,1} = [\cB^{21}, \cA^{22}]
=  - \sum_{m,n \geq 0}\sum_{r\geq 0} A^{22}_{mr}B^{21}_{rn}\psi^2_m\psi^{1*}_{-n}.
\ee
Now we have
\ben
\tilde{V} = \exp(\cA^{11} + \cA^{2112,1}) \langle 0_2| \exp(\cA^{21}) \exp(\cA^*)  \exp(\cA^{12})\exp(\cA^{22}) \exp(\cA^{21,1}) \exp(\cB^{21})|0_{2}\rangle |0_{1}\rangle.
\een
Note $\cB^{21}|0_2\rangle = 0$ and $\langle 0_2| \cA^{21} = 0$,
we have
\be
\tilde{V} = \exp(\cA^{11} + \cA^{2112,1}) \langle 0_2|
\exp(\cA^*)  \exp(\cA^{12})\exp(\cA^{22}) \exp(\cA^{21,1})
|0_{2}\rangle |0_{1}\rangle.
\ee
Similarly, one can show that
\be
\exp(\cA^*) \exp(\cA^{12})
=  \exp(\cA^{12}) \exp(\cA^*)\exp(\cB^{12}),
\ee
where
\ben
\cB^{12} = [\cA^*, \cA^{12}]
= \sum_{m , n \geq 0}B^{12}_{mn}\psi^1_\bm \psi_\bn^{2*}
\een
which commutes with both $\cA^*$ and $\cA^{12}$, where
\be
B^{12}_{mn} = -\sum_{r\geq 0} A^{12}_{mr}Q^{n+r+1}A_{nr}
\ee
are formal power series of $Q$ which are divisible by $Q$.

\be
\exp(\cB^{12}) \exp(\cA^{22}) = \exp(\cA^{22}) \exp (\cA^{12,1}) \exp(\cB^{12}),
\ee
where
\be
\cA^{12,1} = [\cB^{12}, \cA^{22}]
= \sum_{m,n \geq 0} \sum_{r \geq 0} B^{12}_{mr} A^{22}_{rn} \psi_\bm^1 \psi^{2*}_{-\bn},
\ee
and
\be
\exp(\cB^{12}) \exp(\cA^{21,1})
= \exp(\cA^{1221,1})\exp(\cA^{21,1})\exp(\cB^{12}),
\ee
where
\be
\cA^{1221,1} = [\cB^{12}, \cA^{21,1}] = \sum_{m,n\geq 0}( \sum_{r\geq 0}B^{12}_{mr}A^{21,1}_{rn})\psi^1_{\bm}\psi^{1*}_{-\bn}
\ee
commutes with both $\cB^{12}$ and $\cA^{21,1}$.
Because $\langle 0_2| \cA^{12} = 0$ and $\cB^{12}|0_2\rangle=0$,
\ben
&&  \langle 0_2| \exp(\cA^*)  \exp(\cA^{12})\exp(\cA^{22}) \exp(\cA^{21,1}) |0_{2}\rangle  \\
& = &  \langle 0_2| \exp(\cA^{12})\exp(\cA^*)  \exp(\cB^{12})
\exp(\cA^{22}) \exp(\cA^{21,1})|0_{2}\rangle  \\
&  = & \langle 0_2|
\exp(\cA^*)  \exp(\cA^{22}) \exp(\cA^{12,1}) \exp(\cA^{1221,1})\exp(\cA^{21,1})\exp(\cB^{12})
|0_{2}\rangle  \\
&  = & \exp(\cA^{1221,1}) \langle 0_2|
\exp(\cA^*)  \exp(\cA^{22}) \exp(\cA^{12,1})\exp(\cA^{21,1})
|0_{2}\rangle.
\een
Because the operators $\cA^{22}, \cA^{12,1}$ and $\cA^{21,1}$ commute with each other,
we now have
\ben
\tilde{V} =
\exp(\cA^{11}+\cA^{2112,1}+\cA^{1221,1})
\langle 0_2|\exp(\cA^{*})\exp(\cA^{21,1})\exp(\cA^{12,1})\exp(\cA^{22}) |0_2\rangle|0_1\rangle.
\een
Recall $\cA^{2122,1}, \cA^{12,1}, \cA^{21,1}$ are divisible by $Q$,
and $\cA^{1221,1}$ is divisible by $Q^2$.
By repeating the above procedure $N$-times one gets:
\ben
\tilde{V} & = &
\exp(\cA^{11}+\sum_{j=1}^N (\cA^{2112,j}+\cA^{1221,j}) \\
&& \langle 0_2|\exp(\cA^{*})\exp(\cA^{21,N})\exp(\cA^{12,N})\exp(\cA^{22})) |0_2\rangle|0_1\rangle,
\een
where $\cA^{2121,j}, \cA^{12,j}, \cA^{21,j}$
and $\cA^{1221,j}$ is divisible by $Q^j$.
Therefore,
by taking $N \to \infty$,
\be
\tilde{V} = \langle 0_2|\exp(\cA^{*}) \exp(\cA^{22}) |0_2\rangle \cdot
\exp(\cA^{11}+\sum_{j=1}^\infty (\cA^{2112,j}+\cA^{1221,j}) |0_1\rangle.
\ee
This completes the proof of Theorem 4.1.
\end{proof}

\newpage
\newpage
\maketitle
%\tableofcontents

\end{document}